\documentclass[11pt]{amsart}
\usepackage{pstricks,pst-plot}

\definecolor{Black}{cmyk}{0,0,0,1}
\definecolor{OrangeRed}{cmyk}{0,0.6,1,0}            
\definecolor{DarkBlue}{cmyk}{1,1,0,0.20}
\definecolor{myblue}{rgb}{0.66,0.78,1.00}
\definecolor{Violet}{cmyk}{0.79,0.88,0,0}
\definecolor{Lavender}{cmyk}{0,0.48,0,0}

\parskip=\smallskipamount

\newtheorem{theorem}{Theorem}[section]
\newtheorem{lemma}[theorem]{Lemma}

\newtheorem{proposition}[theorem]{Proposition}

\theoremstyle{definition}
\newtheorem{definition}[theorem]{Definition}
\newtheorem{example}[theorem]{Example}

\newtheorem{remark}[theorem]{Remark}

\newcommand{\bea}{\begin{eqnarray*}}
\newcommand{\eea}{\end{eqnarray*}}

\numberwithin{equation}{section}

\begin{document}

\title[ EXTENSION OF LINE BUNDLES]{$Q-$COMPLETE DOMAINS WITH CORNERS IN $\mathbb P^n$ AND EXTENSION OF LINE BUNDLES}
%
%

\author{John Erik Forn\ae ss, Nessim Sibony}
\author{Erlend F. Wold}
\address{J. E. Forn\ae ss, NTNU,  7491 Trondheim, Norway. john.fornass@math.ntnu.no and Mathematical Sciences Center, Tsinghua University}
\address{N. Sibony, Universit\'{e} Paris-Sud, Mathematique, F-91405 Orsay Cedex, France, 
Nessim.Sibony@math.u-psud.fr}
\address{E. F. Wold, Matematisk Institutt, Universitetet i Oslo,
Postboks 1053 Blindern, 0316 Oslo, Norway, erlendfw@math.uio.no}

\footnote{
First author supported by an NSF grant.  Third author supported by the Norwegian 
Research Council.
}

\maketitle

Dedicated to  T. Ohsawa and T. Ueda  for their sixtieth birthday\\

\today

\begin{abstract}
We show that if a compact set $X$ in $\mathbb P^n$ is laminated by holomorphic submanifolds of dimension $q$, then $\mathbb P^n \setminus X$ is $(q+1)$-complete with corners. Consider a manifold $U,$ $q$-complete with corners. Let $\mathcal N$ be a holomorphic line bundle  in the complement of a compact in $U$. We study when $\mathcal N$ extends as a holomorphic line bundle in $U.$ We give applications to the non existence of some Levi-flat foliations  in open sets in $\mathbb P^n$. The results apply in particular when $U$ is a Stein manifold of dimension $n\geq 3$.
 \end{abstract}

\section{Introduction} Let $\mathbb P^n$ be the complex projective space of dimension $n.$ Let $X$ be a compact set laminated by nonsingular varieties of dimension $1\leq q \leq n-1.$ We  prove that $\Omega:=\mathbb P^n\setminus X$ is  $(q+1)$-complete with corners. We then study the extension of holomorphic line bundles defined out of a compact set of a manifold $U,q$-complete with corners. This requires $q$ to be large enough with respect to the dimension of $U.$

We recall a few definitions.

\begin{definition}
A real $\mathcal C^2$ function $u$ defined in a complex manifold $U$ is strictly $q$-convex if the complex Hessian (Levi form) has at least $q$ strictly positive eigenvalues at every point. A complex manifold $U$ is $q$-complete if it admits a strictly $q$-convex exhaustion function.
\end{definition}
By smoothing, this is equivalent to saying that $U$ admits a $\mathcal C^\infty$ $q$-convex
exhaustion function.  
\begin{definition}
A function $\rho$ is strictly $q$-convex with corners on $U$ if for every point $p\in U$ there is a neighborhood $U_p$ and  finitely many strictly $q$-convex $\mathcal C^2$ functions $\{\rho_{p,j}\}_{j \leq \ell_p}$
on $U_p$ such that $\rho_{|U_p}= \max_{j\leq \ell_p} \{\rho_{p,j}\}.$
\end{definition}

\begin{definition}
The manifold $U$ is strictly $q$-convex with corners if it admits an exhaustion function which is strictly $q$-convex with corners outside a compact set.
\end{definition}

\begin{definition}
The manifold $U$ is strictly $q$-complete with corners if it admits an exhaustion function which is strictly $q$-convex with corners.
\end{definition}

The notion of $q$-convexity with corners was introduced by Grauert [1].The  regularization of strictly $q$-convex functions with corners were studied by Diederich-Forn\ae ss [4]. We will also use the following result by M. Peternell [12].

\begin{theorem}
If $U$ is an open set in $\mathbb P^n$ which is strictly $q$-convex with corners, then it is strictly $q$-complete with corners.
\end{theorem}
M. Peternell proved also that the complement in $\mathbb P^n$
of a (possibly singular) analytic set of dimension $q$ is $(q+1)$-complete with corners.

We recall the notion of compact laminated set. A compact set $X$ in a complex manifold $U$ is laminated by complex manifolds of dimension $q$ if there is an atlas $\mathcal L$ on $X$ with charts $\Phi_V:V \rightarrow \mathbb B \times T_V$ where $\mathbb B$ is a $q$ dimensional ball in $\mathbb C^q,$  $T_V$ is a topological space,  $\Phi_V$ is a homeomorphism. The change of coordinates are of the form $(z,t)\rightarrow (z',t')$ with $z'=h(z,t), t'=t'(t).$ The map $h$ is holomorphic in $z.$ 
If $X=U$ and the change of coordinates $h,t'$ are holomorphic we say that we have a holomorphic foliation of dimension $q.$ In $\mathbb P^n$ it is a classical result that holomorphic foliations are always singular. More precisely, there is a nonempty analytic set $E$ of codimension at least two so that $U:=\mathbb P^n \setminus E$ is foliated as described above. One can also consider a covering $(U_i)$ of $\mathbb P^n$ and in each $(U_i)$ we have $(w_i^1, \dots,w_i^{n-q})$ holomorphic $(1,0)$ forms defining the leaves, satisfying the Frobenius condition and such that on $U_i\cap U_j$ there is an invertible holomorphic matrix $(g_{ij})$ with $(w_i)=(g_{ij})(w_j).$ 

The main results are the following:

\begin{theorem}
Let $X$ be a compact set laminated  by complex manifolds of dimension $1\leq q\leq n-1$ in $\mathbb P^n.$ Then $\Omega:=\mathbb P^n \setminus X$ is strictly $(q+1)$-complete with corners.
\end{theorem}

\begin{theorem}
Let $X$ be a complex manifold with a proper Morse function $\rho:X\rightarrow (a,b)$
whose Levi form has at least three strictly positive eigenvalues at each point of $X$.  
Then every holomorphic line bundle over a set 
$$
\{x\in X: \rho(x)>c\} \mbox{ for some } c\in(a,b)
$$
extends to a holomorphic line bundle over $X$.  The same holds if $(a,b)$
is replaced by $[a,b)$.
\end{theorem}

The smoothing techniques by the first author and Diederich \cite{DiederichFornaess} 
imply then the following: 
\begin{theorem}
Let $U$ be a complex manifold of dimension $n\geq 3$. Assume $U$ is $q$-complete with corners with a corresponding exhaustion function $\rho.$ Let $\mathcal N$ be a holomorphic line bundle defined in $\{\rho>c\}.$ Assume $q \geq  \frac{2n+3}{3}.$
Then the line bundle $\mathcal N$ can be uniquely holomorphically extended to $U$.
\end{theorem}

 Our interest in the problem of extending a holomorphic line bundle has its origin in a paper by T. Ohsawa [11]
on the nonexistence of real analytic Levi flat hypersurfaces in $\mathbb P^2$, although 
he considered the extension of \emph{roots} of globally defined bundles.
 If it were possible to extend roots of line bundles in dimension $2$, a strategy by T. Ohsawa would give the nonexistence
of real analytic Levi flat hypersurfaces in $\mathbb P^2.$

Indeed if $M$ is a real analytic Levi flat hypersurface in $\mathbb P^2$, then the foliation on $M$ extends to a foliation in a neighborhood of $M$ and by pseudoconvexity of the complement we will obtain a foliation $\mathcal F$ of $\mathbb P^2.$ If $\mathcal N$ denotes the normal bundle of the foliation it is easy to show that it has roots $\mathcal N_k$ for all $k \in \mathcal N$ in a given neighborhood of $M.$ If it were possible to extend them to $\mathbb P^2$, the Chern class of $\mathcal N$ would be zero, a contradiction.

After the above results were obtained, S. Ivashkovich in an arXiv preprint   \cite{Ivashkovitch}
has obtained Hartogs' extension of roots of a line bundle in dimension $n\geq 3$,  assuming that $\rho$ is a strictly plurisubharmonic exhaustion function.

In the last section we give an application of the above results to the 
non-existence of certain Levi-flat manifolds in open subset of projective spaces. 
There are however C.R.-manifolds in $\mathbb P^n$ which are laminated by complex manifolds. Several examples are discussed in Ghys [7] and in Forn\ae ss-Sibony-Wold [5].

The following question seems open. Let $\mathcal F$ be a holomorphic foliation by Riemann surfaces in $\mathbb P^n.$ It always has singularities. Is there a closed nonempty $\mathcal F$ invariant set which does not intersect the singularities of $\mathcal F?$ Equivalently,  does every leaf cluster on ${\mbox{Sing}}(\mathcal F)$? The question is discussed for $n=2$ in Camacho-Lins Neto-Sad [3]. But the answer seems unknown even for all $n\geq 3.$

\medskip

\emph{ACKNOWLEDGEMENT:} We would like to thank the referee for many useful comments.

\section{Extension of line bundles across C.R. manifolds.}

We recall some basics about line bundles.
A topological (resp. holomorphic) line bundle $L\rightarrow X$ over 
a complex manifold $X$ is an element of $H^1(X,\mathcal C^*)$ (resp. $H^1(X,\mathcal O^*)$), 
\emph{i.e.}, $L$ is represented by a cover $\{U_j\}$ of $X$ and 
functions $f_{ij}\in\mathcal C^*(U_i\cap U_j)$ (resp. $f_{ij}\in\mathcal O^*(U_i\cap U_j)$), 
satisfying the cocycle condition $f_{ij}\cdot f_{jk}\cdot f_{ki}=1$. 
The bundle $L$ is trivial if there exist $f_i\in\mathcal C^*(U_i)$ (resp. $f_i\in\mathcal O^*(U_i)$)
such that $f_{ij}=f_j/f_i$.  A way to show that $L$ is trivial is to 
choose branches $g_{ij}=\log f_{ij}$ such that the cocycle 
condition $g_{ij} + g_{jk} + g_{ki} = 0$ is satisfied, solve $g_{ij} = g_j - g_i$ (which is
always possible in the topological category), 
and define $f_i:=e^{g_i}$.  Of course, branches of logarithms can be chosen 
in accordance with the cocycle condition only if the bundle is trivial. Recall that
the short exact sequence  
$$
0\rightarrow \mathbb Z\rightarrow\mathcal C\overset{\mathrm{exp}}{\rightarrow}\mathcal C^*\rightarrow 0
$$
induces an isomorphism $H^1(X,\mathcal C^*)\approx H^2(X,\mathbb Z)$.

If $X$ deformation retracts to a closed subset $X'\subset X$ with 
$L|_{X'}$ trivial and $L$ is a topological line bundle, then $L$ is trivial, since branches of logarithms 
may then be chosen on $X'$.  

If $\{U_j\}$ is a cover of 
$X$ with all line bundles trivial over $U_j$, then any line bundle 
has a representation with respect to this cover.    
Assuming that a holomorphic bundle $L$ is topologically 
trivial, we may split $f_{ij}=f_j/f_i$ with smooth $f_j$'s.  Define 
$g_j:=\log f_j$ (choose any branch).  Then $g_j-g_i$ is holomorphic 
on $U_i\cap U_j$, hence $\omega:=\overline\partial g_j$ 
is a well defined $(0,1)$-form.  So if we can solve $\overline\partial$
at the level of $(0,1)$-forms, we solve $\overline\partial f=\omega$, 
and define $\tilde g_j = g_j - f$ and note that $f_j:=e^{\tilde g_j}$
defines a holomorphic splitting of the 1-cocycle, \emph{i.e.}, $L$ 
is holomorphically trivial.

Our approach to proving Hartogs' extension for line bundles on a 
3-convex $n$-dimensional manifold $X$, will be to choose a suitable 3-convex exhaustion 
$\rho$ of $X$ with Morse critical points, start with a line bundle $L\rightarrow\{\rho>c\}$,  and 
extend locally across the level sets of $\rho$ until we cover $X$.  
This basically reduces to showing that $L$ is locally trivial near 
any point $p\in\{\rho=c\}$.  In all cases except for the case where $p$   
is a critical point with index $2n-3$, we observe that any 
\emph{topological} line bundle is 
trivial near $p$. Then we use the fact, see Theorem 2.3  \cite{Andreotti-Grauert} below,  that we can solve the $\overline\partial$-equation
at the level of $(0,1)$-forms near $p$.
Near a critical point with index $2n-3$, non-trivial topological bundles do exist.
So in this case the triviality is an analytic phenomenon.  The 
key difficulty is to prove Hartogs' extension phenomenon for line bundles 
across real analytic totally real manifolds of dimension three, and so we state this 
as a separate result.

\begin{theorem}\label{main}
Let $X$ be a complex manifold of dimension at least three and let $M\subset X$ be a closed 
real analytic CR submanifold of CR-dimension less than or equal to $n-3$.
Then the restriction map $H^1(X,\mathcal O^*)\rightarrow H^1(X\setminus M,\mathcal O^*)$ is surjective.  
\end{theorem}

\begin{remark}
It follows from the previous theorem that an analytic set M, of codimension 
at least three, is removable for holomorphic line bundles. We can first extend through the smooth points of M and then use the stratification of M into smooth manifolds.
\end{remark}

\begin{remark}
Ivashkovich \cite{Ivashkovitch} has given a counter example to extension for 
a two dimensional real analytic totally real manifold in $\mathbb C^2$, even under the assumption that 
the bundle is topologically trivial.  We give other examples at the end of this section and in Section 5.
\end{remark}

We rely on the following result by Andreotti-Grauert [1, Proposition 12]. We have translated their result using our definition of $q$-convexity. 

\begin{theorem}\label{AG}
Let $V$ be a domain in $\mathbb C^n, n \geq 3$  and let $p\in V$. Let $\rho$ be a smooth strictly q-convex function on V, with $q \geq 3$. There is a basis $Q$ of Stein neighborhoods of p such that  on the intersection $D$ of $Q $ with $\{\rho>c=\rho(p)\},$ the $\overline{\partial}$-equation is solvable for $(0,1)$-forms. 
More precisely, for any $\overline{\partial}$-closed smooth $(0,1)$ form $f$  in $D$ there is a smooth function $u$ in $D$ such that $\overline{\partial}u=f.$
\end{theorem}

\medskip

To prove Theorem \ref{main} it is enough to show the following: for any point $x\in M$ there exists 
an open neighborhood $U_x$ of $x$ in $X$ such that any holomorphic
line bundle on $U_x\cap (X\setminus M)$ is holomorphically trivial. 
In that case, given a line bundle $L\rightarrow X\setminus M$, we may 
find an open cover $\{U_\alpha\}_{\alpha\in I}$ of $X$ such that $L$
is trivial over $U_\alpha\setminus M$ for all $\alpha$.  
Then $L$ is represented by a cocycle $f_{\alpha,\beta}\in\mathcal O^*(U_{\alpha,\beta}\setminus M)$.  But any holomorphic function in $U_\alpha\setminus M$ extends holomorphically across $M$, 
and so by the identity theorem $f_{\alpha,\beta}$ extends to a cocycle 
with respect to the cover $\{U_\alpha\}_{\alpha\in I}$ of $X$.   \

At any point of $M$ we may choose local holomorphic coordinates $(z,w)\in\mathbb C^r\times\mathbb C^s$ on $X$, $z_j=x_j+iy_j,$ 
such that $M=\{z_j=F_j,j=1,\cdots,k, x_j=G_j, j=1,\cdots,r\}$
where the $F_j,G_j$ are CR functions of $y_{k+1},\cdots, y_r, w_1,\cdots w_s$, holomorphic in $w$ and, vanishing to second order. The CR-dimension of $M$ is $s$, and the real codimension $r\geq 3$. The function $\sum_j |z_j-F_j|^2+\sum_j |x_j-G_j|^2$ is $q$-convex. \

In the case $r\geq 4$ or if $r=3$ and $k>0$, the proof is short: near any point $x\in M$
we have that $L$ is topologically trivial; $H^2((X\setminus M)\cap U_p,\mathbb Z)$ is zero because $X\setminus M$ retracts on a ball $B^*$ of dimension $>3$.  By Theorem \ref{AG} we may also 
solve $\overline\partial$ near $x$, and so $L$ is holomorphically trivial.  \

The main difficulty is to prove the result if $s=0$, $r=3$, and $k=0$, \emph{i.e.}, if
$M$ is a three dimensional totally real manifold in complex three dimensional space.
If $s$ were greater than zero then $\mathbb C^{r+s}\setminus M$ would (locally)
retract onto $\Omega:=\mathbb C^3\setminus\{M\cap\{z_{r+1}=\cdot\cdot\cdot=z_{r+s}=0\}\}$.
A line bundle on $\mathbb C^{r+s}\setminus M$ restricted to $\Omega$, is locally topologically 
trivial by the extension result for $s=0$, hence it is locally topologically trivial on $\mathbb C^{r+s}\setminus M$.

So we need to prove Theorem \ref{main} in the case where 
$dim(X)=3$, and $M$ is a totally real three dimensional manifold.  
Let $x\in M$.  Since $M$ is real analytic there exists an open neighborhood $U_x$ of $x$ in $X$ and 
a biholomorphism $\psi:U_x\rightarrow V_0\subset\mathbb C^3_0$ such that 
$
\psi(M)=\tilde M=\{z\in V_0: Re(z_j)=0 \mbox{ for } j=1,...,3\}.
$
So to prove Theorem \ref{main}
it suffices to prove the following: 

\begin{proposition}\label{extension}
Let $L\rightarrow V_0\setminus\tilde M$ be a holomorphic line bundle.  Then there exists an 
$\epsilon>0$ such that $L$ is holomorphically trivial over $(V_0\setminus\tilde M)\cap\mathbb B^3_{\epsilon}$.
\end{proposition}

Let $B\subset\mathbb R^3$ denote the open unit ball, and for $j=1,2$ let $B_j$
denote the ball of radius one centered at $((-1)^{j+1},0,0)$.   Let 
$\tilde U_j:=B\setminus\overline{B_j}$.   Let $D$ denote the disk 
$D:=\{(x_1,...,x_3)\in B:x_1=0\}$, and note that 
 $\tilde U_1\cap\tilde U_2$
retracts onto the punctured disk $D^*$.   Let $S^1$ denote the circle of radius one half 
contained in $D^*$.

For $j=1,2$ let 
$U_j=(\tilde U_j\times i\cdot\mathbb R^3)\cap\mathbb B^3\subset\mathbb C^3$, 
and let $D^*\hookrightarrow U_1\cap U_2$ be the natural inclusion.  
Then $U_1\cap U_2$ retracts onto $D^*$.
Now $(U_1,U_2)$ is an open cover of $\mathbb B^3\setminus\tilde M$.
Since each $U_j$ retracts to a point, any topological line bundle 
$L\rightarrow\mathbb B^3\setminus\tilde M$ is represented by a function $f_L\in\mathcal C^*(U_1\cap U_2)$.

\begin{lemma}\label{winding}
The line bundle $L$ is trivial if and only if the winding number of $f_L$ restricted
to $S^1$ is zero.  
\end{lemma}
\begin{proof}
Since $U_1\cap U_2$ retracts onto $D^*$, the function $f_L$ has 
a continuous logarithm $\log f_L$ if the winding number is zero. 
Since $H^1(\mathbb B^3\setminus\tilde M,\mathcal C)=0$
we may write $\log f_L=g_2-g_1$ for $f_j\in\mathcal C(U_j)$.
So $f_j:=e^{g_j}$ defines a multiplicative splitting of $f_L$. 

On the other hand, assume that $f_L=f_2/f_1$ for $f_j\in\mathcal C^*(U_j)$.
Since each $U_j$ is star-shaped it follows that each $f_j$ is homotopic 
through non-zero maps to a constant map, hence $f_L$ is homotopic
through non-zero maps to a constant map.  
\end{proof}

\begin{lemma}\label{trivialcover}
There exists an $\epsilon>0$ such that if $L\rightarrow U_j\cap\mathbb B^3$ is 
a holomorphic line bundle, then $L|_{U_j\cap\mathbb B^3_{\epsilon}}$
is holomorphically trivial.  
\end{lemma}
\begin{proof}
By Theorem \ref{AG} there is an 
open neighborhood $\Omega_j\subset\mathbb B^3$ 
of the origin, such that the $\overline\partial$-equation
for $(0,1)$-forms
is solvable on $\Omega_j\cap U_j$.
Since $U_j\cap\mathbb B^3$ is star-shaped, the obstruction to $L$ being 
trivial on $U_j\cap\mathbb B^3$ is represented by a $\overline\partial$-closed $(0,1)$-form 
$\omega$.  Since $\omega$ is $\overline\partial$-exact on $\Omega_j\cap U_j$
we have that $L|_{\Omega_j\cap U_j}$ is holomorphically trivial.  
Choose $\epsilon>0$ small enough such that $\mathbb B^3_{\epsilon}\subset\Omega_j$
for $j=1,2$.
\end{proof}

\begin{lemma}\label{debarcohom}
There exists a neighborhood $U$ of 
the origin such that the following holds: let $\omega$ be a $\overline\partial$-closed (0,1)-form on $\mathbb B^3\setminus\tilde M$.  
Then $\omega$ is $\overline\partial$-exact on $U\setminus\tilde M$.

\end{lemma}
\begin{proof}
This follows from Theorem \ref{AG} and using the function 
$$
\rho(z):=\sum_j (Re(z_j))^2.
$$

\emph{Proof of Proposition \ref{extension}:} 
In the case $dim(X)=k=3$.

Assume to get a contradiction that 
$L\rightarrow \mathbb B^3\setminus\tilde M$ is a holomorphic line bundle, which is 
not trivial in any neighborhood of the origin.  
From Lemma \ref{debarcohom}  we see that $L$ is not topologically trivial, 
otherwise $L$ would be holomorphically trivial on $\mathbb B^3_{\epsilon_0}\setminus\tilde M$.    By Lemma \ref{trivialcover} 
there exists an $\epsilon>0$ such that  
$L|_{\mathbb B^3_\epsilon\setminus\tilde M}$ is represented by a non-vanishing holomorphic function $f\in\mathcal O^*((U_1\cap U_2)\cap\mathbb B_\epsilon)$.

Since $L$ is topologically nontrivial  we get from Lemma \ref{winding} (and scaling)
that $f$ has a non-zero winding number when restricted to $S^1$.
Restricting $f$ to $\{z_1=0\}\cap (\mathbb B_\epsilon^3\setminus\tilde M)$
we get by Hartogs' extension phenomenon that $f$ extends 
to $\{z_1=0\}\cap\mathbb B^3_\epsilon$.  Then non-trivial winding implies 
that the extended $f$ must have zeroes.  On the other hand these 
zeroes must be contained in $\tilde M$, which is a contradiction.  

\end{proof}

\medskip
\begin{example}
We give an example of a 
topologically trivial line bundles with
roots
 in the complement of a real line in $\mathbb C^2$, which is not
holomorphically trivial:

 Cover the complement of the real z-axis by two open sets:
 $U= \{(z,w); w \neq 0\}\cup \{(x+iy,0); y<0\}$ and $V=\{(z,w); y>0\}.$
 We define the line bundle by using the transition function $e^{1/w}$
 on $U \cap V=\{(z,w); y>0, w\neq 0\}.$ Since we can write $1/w$ as the difference
 between smooth functions on $U$ and $V$, it follows that the line bundle is topologically trivial.
 If the line bundle were holomorphically trivial, then $e^{1/w}=\frac{f_U}{f_V}$. But $f_U$ extends holomorphically to a neighborhood of $0$ so neither $f_U$ nor $f_V$ can be singular at $w=0,$
 a contradiction.
\end{example}

\section{Proof of Theorem 1.7}

Let $\rho$ be a smooth exhaustion of a domain $U$ in $\mathbb P^n.$ We want to extend a holomorphic line bundle defined in $(\rho>c)$ past $p$ with $\rho(p)=c$. We have to show that locally in $(\rho>c)$ the line bundle is holomorphically trivial and then extend the chart near $p.$  We reduce to a $\overline{\partial}$-problem for $(0,1)$-forms in the intersection of a neighborhood of $p$ with 
$(\rho>c).$ For that we need that the Levi form at $p$ has at least three strictly positive eigenvalues and we need to construct an appropriate neighborhood.  \

We will be interested in extending holomorphic line bundles across 
certain (singular) hypersurfaces.  The following lemma
gives some conditions that easily imply extension: 

\begin{lemma}\label{extension2}
Let $K$ be a closed subset of a complex manifold $X$
with $K=\overline{K^\circ}$.  Assume that there 
exist open covers $\{U_j'\}\ll\{U_j\}$ of $bK$ such 
that the following holds: 
\begin{itemize}
\item[(1)] Any holomorphic line bundle $L\rightarrow K^\circ$
is trivial over $U_j\cap K^\circ$, 
\item[(2)] any $f\in\mathcal O(U_i\cap U_j\cap K^\circ)$ extends 
to a holomorphic function on $U'_i\cap U'_j$ for $i\neq j$, and
\item[(3)] each connected component of $U_i'\cap U_j'\cap U'_k$
intersects $K^\circ$ for all $i,j,k$.
\end{itemize}
Then any holomorphic line bundle on $K^\circ$ extends 
to a holomorphic line bundle on $K\cup (\underset{j}{\cup} U'_j)$. 
\end{lemma}

\begin{proof}
Let $L\rightarrow K^\circ$ be a holomorphic line bundle.   
Let $\{V_j\}$ be a collection of open subsets of $K^\circ$
with $H^1(V_j,\mathcal O^*)=0$, and such that 
$\{V_j\}\cup\{U'_j\cap K^\circ\}$ is an open cover of $K^\circ$. 
Then by (1) the line bundle $L$ is represented by 
transition functions on the intersections of all pairs of 
open sets in this cover.  By (2) each transition function 
$f'_{ij}\in U_i'\cap U_j'\cap K^\circ$ extends to
a function $f_{ij}$ on  $U_j'\cap U_i$, and since $1/f'_{ij}$ 
also extends, $f_{ij}$ is non vanishing.  Now we are 
done if 
\begin{equation}\label{cocycle}
f_{ij}\cdot f_{jk}\cdot f_{ki}=1 \mbox{ on } U_i'\cap U_j'\cap U_k'
\mbox{ for all } i,j,k.  
\end{equation}

But this follows from (3) and the identity principle, since 
the cocycle condition (\ref{cocycle}) is satisfied on $U_i'\cap U_j'\cap U_k'\cap K^\circ$. 
\end{proof}

\begin{lemma}\label{smoothboundary}
Let $X$ be a complex manifold, let $\rho$ 
be a nice strongly $q$-convex exhaustion of $X$, $q\geq 3$, 
set $\Omega_c:=\{\rho > c\}$, and assume that $b\Omega$
does not contain any critical points for $\rho$. Then 
there exists $c'<c$ such that any holomorphic 
line bundle on $\Omega_c$ extends to a holomorphic 
line bundle on $\Omega_{c'}$.
\end{lemma}
\begin{proof} 
Construct a cover as in Lemma \ref{extension2}
by using 
Hartogs'  extension phenomenon; Theorem 13.8
in \cite{HenkinLeiterer2}, and Theorem \ref{AG}.
\end{proof}

We start with a brief discussion about the local structure of Morse exhaustion functions. 
If $p\in X$ is a Morse critical point for a strictly  $q$-convex function 
$\rho$, then there are local holomorphic coordinates 
$$
z=(z_1,...,z_q,z_{q+1},...,z_n):U_p\rightarrow\mathbb C^n,
$$
with $z(p)=0$, such that the following holds (see \cite{Forstneric}, Lemma 3.9.4, or \cite{Forstneric2})
$$
\rho(z',0)=\sum_{j=1}^{q-r} x_j^2 + \delta_j y_j^2 + \sum_{j=q-r+1}^q x_j^2 - \delta_jy_j^2 + o(|z'|^2), 
$$
with $(z',z'')\in\mathbb C^q\times\mathbb C^{n-q}$, $\delta_j>0$, $\delta_j<1$
for $j=q-r+1,...,q$, and $r$ is the Morse index of the function $\rho(z',0)$.   

\medskip

Furthermore
there is a real coordinate change $\psi(z',z'')$, holomorphic for each fixed 
$z''$, such that $\tilde\rho(z)=\rho(\psi(z))$ is of the form $\rho(\psi(z))=Q(z) + o(|\psi(z)|^2)$ with 
$$
Q(z)= \sum_{j=1}^{q-r} x_j^2 + \delta_j y_j^2 + \sum_{j=q-r+1}^q x_j^2 - \delta_jy_j^2 - 
\sum_{j=1}^{m} u_j^2 + \sum_{j=m+1}^{2(n-q)} u_j^2,
$$
where the $u_j$'s are real coordinates on $\mathbb C^{n-q}$, and $r+m$ is 
the Morse index of $\rho$ at $p$. 

We may further modify $\tilde\rho$ by choosing a $\chi\in\mathcal C^\infty_0([0,1])$ which is 
0 near the origin and 1 near 1.We then define 
$$
\tilde\rho_r(z) := Q(z) + \chi(\|z\|/r)o(\|z\|^2), 
$$
for $0<r<<1$.  Then $\tilde\rho_r$ converges to 
$\tilde\rho$ in $\mathcal C^2$-norm as $r\rightarrow 0$, 
since the remainder vanishes to order three, $\tilde\rho_r$ does 
not have more critical points than $\tilde\rho$, and 
$\tilde\rho(z)=Q(z)$
near the origin.  Such critical points, where $\tilde\rho_r$ is \emph{equal} to its second
order expansion, will be called  \emph{nice} critical points.   By our discussion, 
any $q$-complete manifold admits a smooth exhaustion function with only 
nice critical points.  

\medskip

Finally, by perturbing $\rho$ near a critical point we may assume that the $\delta_j$'s 
are equal to one for $j=1,...,q-r$ and as small as we like for $j=q-r+1,...,q$.
To see this, let $\chi$ be a cut off function with support in a neighborhood of $0$. 
Define

$$
\tilde{\rho}=\rho+ \chi (\frac{z}{r}) \sum_{j\leq q} \alpha_j y_j^2.
$$

Here $|\alpha_j|$ are small. 
It is easy to check that $\tilde{\rho}$ has no critical points except $0$ if $|\alpha_s|$ is not too large.  Near zero we have improved the $\delta_j$'s.  The constants $\alpha_j$ are controlled by how large the positive eigenvalues of $\rho$ are where $\chi$ varies. We can continue the process, with possibly smaller $r$'s, in finitely many steps.

\begin{lemma}\label{trivialcritical}
Let $X$ be a Stein manifold and let $\rho$ be a nice
strongly plurisubharmonic exhaustion function. 
Let $p\in X$ be a critical point for $\rho$.  Then there exists an open neighborhood 
$U_p$ of $p$ in $X$ such that any holomorphic line bundle 
$L\rightarrow \{\rho>\rho(p)\}$ is trivial over $U_p\cap\{\rho >\rho(p)\}$.

\end{lemma}
\begin{proof}
In local coordinates, say on the unit ball $\mathbb B^n$ in $\mathbb C^n$, we may write 
$$
\rho(z)=\sum_{j=1}^{n_1} x_j^2 + y_j^2 + \sum_{j=n_1+1}^n x_j^2 - \delta_j y_j^2, 
$$
where we have normalized and assumed that $\rho(p)=\rho(0)=0$.  We 
know from Theorem \ref{AG} that we may solve $\overline\partial$
for $(0,1)$-forms on $\{\rho>0\}\cap U$ where $U$ is some neighborhood 
of the origin in $\mathbb C^n$, and so the main point is to show 
that $L$ is topologically trivial on $\{\rho>0\}\cap U$.  We see 
that $\{\rho>0\}\cap\mathbb B^n$ retracts onto a $(2\cdot n_1 + n -1)$-sphere, 
and so the only case we need to consider is if $n=3$ and $n_1=0$. \

Choose a suitable smooth function $\chi$ 
which vanishes near the origin, is increasing  and equal to  $1$ for $t>\delta$ for some small $\delta$.
Then  defining $\tilde\rho(z):=\sum_{j=1}^3 x_j^2 - \delta_j\cdot\chi(\|z\|^2)\cdot y_j^2$, 
we have that $\tilde\rho$ has no critical points and 
is strictly plurisubharmonic
on $\chi>0$ (recall that the $\delta_j$'s can be assumed to be as small as we like).   Note that $\tilde\rho$ extends to $X\setminus\mathbb B^3$ since 
$\tilde\rho$ is equal to $\rho$ near $b\mathbb B^3$.  (Here the reader could also 
consult Section 3.10 of \cite{Forstneric})  \

To prove that $L$ is (locally near $p$) topologically trivial it is enough to 
prove that $L$ is topologically trivial when restricted to 
$$
T:=\{z\in\mathbb B^3:|x|^2>1/2, |y|=0\}. 
$$
(Since $\{\rho>0\}\cap\mathbb B^3$ retracts to $T$.)
We start with the line bundle $\tilde L:=L|_{\{\tilde\rho>1/2\}}$ (we look 
at the extension of $\tilde\rho$ to $X$).   We first 
claim that $\tilde L$ extends holomorphically to $\{\tilde\rho>0\}$. 
This follows from Lemma \ref{smoothboundary} by considering 
the set 
$$
\mathcal S:=\underset{s<1/2}{\inf} \{\tilde L \mbox{ extends uniquely to } \{\tilde\rho>s\}\}.
$$
But the set $\tilde\rho=0$ is a real analytic totally real manifold 
near the origin, and so by Proposition \ref{extension} we have that $\tilde L$
extends to a neighborhood of the origin.   This shows 
that $\tilde L$ is topologically trivial on some tube 
$T_\delta:=\{z\in\mathbb B^3:|y^2|<\delta\}, $
hence $L$ is trivial on $T$.  
\end{proof}

We are now ready to prove Theorem 1.7 in the special case where $X$ is Stein. 

\begin{theorem}\label{Stein}
Let $X$ be a Stein manifold of dimension 
greater than or equal to three, and let $\rho$ be a strongly plurisubharmonic exhaustion of $X$. 
Then any holomorphic line bundle $L\rightarrow \{\rho>c\}$ extends uniquely to 
a holomorphic line bundle on $X$.  
\end{theorem}

\begin{proof}

We may assume that $\rho$ is a \emph{nice} strongly plurisubharmonic exhaustion. 
Assume for simplicity that $\rho$ has a single minimum point $x$ with $\rho(x)=0$.
We set
$$
s:=\inf\{c'<c:\mbox{ L extends uniquely to } \{\rho>c'\}\},
$$
and claim that $s=0$.   Otherwise it follows by Lemma \ref{smoothboundary}
that $\{\rho=s\}$ contains a critical point $p$, and in local coordinates we 
may write 
$$
\rho(z)=\rho(p) + \sum_{j=1}^{n_1} (x_j^2 + y_j^2) + \sum_{j=n_1+1}^n (x_j^2 - \delta_j\cdot y_j^2).
$$
By Lemma \ref{trivialcritical} there exists an $r>0$ such that $L$ extends 
to a line bundle on $\{\rho>s\}\cap\mathbb B_r$.  Let $\chi$ be a strictly 
positive function which is 1 near the origin and compactly supported on $[0,1)$, 
and consider 
$$
\rho_{\delta}(z) = \rho(z) - \delta\cdot\chi(\|z\|^2/r^2). 
$$
If $\delta$ is chosen small enough we have that $\rho_\delta$ is strictly 
plurisubharmonic and that the origin is the only critical point of $\rho_\delta$ in $\mathbb B_r$.
Now $L$ lives on $\{\rho_\delta>0\}$.   
By Lemma \ref{smoothboundary} we have that $L$ extends to 
a smooth domain $\{\rho_\delta>c''\}\supset\{\rho>c'\}$ with $c'<s$.  This contradicts 
the assumption that $s>0$.
Uniqueness follow from Lemma \ref{uniqueness}.
\end{proof}

\begin{lemma}\label{uniqueness}
Let $X$ be a complex manifold with a strongly 
$q$-convex exhaustion function $\rho$ with $q\geq 2$,
and let $L\rightarrow X$ be a holomorphic line bundle.
If $L$ is trivial over $\Omega_c=\{\rho>c\}$ then $L$ is trivial. 
Consequently, if $\tilde L\rightarrow\{\rho>c\}$ is a holomorphic
line bundle, then there is a \emph{unique} extension 
of $\tilde L$ to $X$, if such an extension exists, 
and if $L_k\rightarrow X$ satisfies $L_k^{\otimes k}=L$
on $\{\rho>c\}$ then $L_k^{\otimes k}=L$.
\end{lemma}
\begin{proof}
We may assume that $\rho$ is Morse with only 
nice critical points. 
We may assume that $b\Omega_c=\{\rho=c\}$ is smooth.  
Let $U_0=\Omega_c$ and let $U_1,...,U_m$ be 
a cover of $b\Omega_c$ similar to that of Lemma \ref{extension2}. 
Then for $c'<c$ close enough to $c$ we have that 
$L\rightarrow\Omega_{c'}$ is represented by 
a cocyle $(f_{ij})$ with respect to this cover.  
The cocycle splits over $\Omega_c$ by assumption, 
and the splitting extends to $\Omega_{c'}$ for 
$c'$ close to $c$.

Let $c_0$ be the largest $c_0<c$ such that $b\Omega_{c_0}$
contains a critical point $p$.  The argument 
above implies that $L$ is trivial over $\Omega_{c_0}$.

Let $U_1=\{\rho>c_0\}$.
By Theorem 13.8
in \cite{HenkinLeiterer2} there is a neighborhood
$U_2$ of $p$ such that any holomorphic function 
on $U_1\cap U_2$ extends to $U_2$.  This shows 
that $L$ is trivial over $\Omega_{c_0}\cup U_2$.
Next modify $\rho$ as in the proof of Theorem \ref{Stein}, 
and we may use the above argument to show that 
$L$ is trivial until we reach the next critical value. 
\end{proof}

\begin{theorem} Let $U$ be a complex manifold of dimension $n.$ Assume $\rho$ is a 
smooth exhaustion function which is strictly $q$-convex with $q \geq 3.$ Let $\mathcal N$ be a holomorphic line bundle defined in $\{\rho>c\}.$ Then $\mathcal N$ extends uniquely to a holomorphic line bundle on $U.$
\end{theorem}

\begin{proof}
We can assume that $\rho$ is a Morse function. A change with respect to the previous proof is that at a critical point $p$ for $\rho$, there are \emph{real} coordinates $z_j=x_j+iy_j$ such that 
$$
\rho= \sum_{s=1}^{n_1} (x_s^2+y_s^2)+\sum_{n_1+1}^{n_2} (x_s^2-\delta_s y_s^2)+\sum_{n_2+1}^n a_kx_k^2 + b_ky_k^2,
$$
\noindent with $n_2 \geq q \geq 3$, $\delta_s<1$, and $a_k,b_k=\pm 1$.  The only difference with respect to the previous proof 
is that we need to prove that any line bundle is trivial near such a critical point.  If 
the number of positive eigenvalues is greater than three, 
it follows that 
any bundle is topologically trivial, and since we can solve $\overline\partial$ we get 
the desired result.  \

Otherwise we have in local coordinates, 
$$
\tilde\rho(z)=\sum_{s=1}^3 x_s^2 - \delta_jy_s^2 - \sum_{j=4}^n (x_s^2 + y_s^2). 
$$
We want to show that $L$ restricted to $\{\rho>0\}\cap\{z_{4}=\cdot\cdot\cdot=z_n=0\}$ is topologically 
trivial near the origin, because $\{\rho>0\}$ retracts to this set near the origin.  \

Choose a function $\chi$ as in the proof of Theorem 3.4 and define 
$$
\rho(z) = \rho(z)=\sum_{s=1}^3 x_s^2 - \chi(\|z\|)\delta_jy_s^2 - \sum_{j=4}^n (x_s^2 + y_s^2), 
$$
such that $\tilde\rho$ has no critical points outside $\chi=0$ and such that 
$\tilde\rho$ is strongly $q$-convex.   Again $\tilde\rho$ extends to $X$.  
We may now extend $L$ restricted to $\{\tilde\rho > c > 0\}$ to 
a bundle $\tilde L$ on $\{\tilde\rho>0\}$.  But now $\tilde L$ restricted to 
$\{\tilde\rho>0\}\cap\{z_4=\cdot\cdot\cdot = z_n=0\}$ is a holomorphic line bundle 
in the complement of a totally real real analytic three dimensional manifold 
near the origin, hence the restriction of $\tilde L$ extends across.  Recall 
that the coordinate change $\psi^{-1}$ from the discussion following Theorem \ref{AG} is holomorphic 
when restricted to $\{z_4=\cdot\cdot\cdot=z_n=0\}$.   As in the proof of Theorem 
3.4 this implies that the original $L$ was trivial near the origin on $\{z_4=\cdot\cdot\cdot=z_n=0\}$.

\end{proof}

We can then prove Theorem 1.8 if we rely on the smoothing of $q$-convex functions with corners obtained by the first author and K. Diederich [4]. We first recall their notation and statement.

Suppose we have a strictly $q$-convex function with corners in the meaning of the present paper.
So if $Q$ denotes the $q$ in [4],
 then the functions have $n-Q+1$ strictly positive eigenvalues.
 Hence $q=n-Q+1, Q=n-q+1$.
 After smoothing we are left with a $\tilde{Q}$-convex function with
 $$
 \tilde{Q}= n- \left[ \frac{n}{Q}\right] +1
 $$
 So we want to find $\tilde{q}$:
 $$
 \tilde{q}= n-\tilde{Q}+1.
 $$
 So 
 $$
 \tilde{q}= n-\left(n- \left[ \frac{n}{Q}\right] +1 \right)+1
 $$
 
 $$
 \tilde{q} = \left[ \frac{n}{Q}\right]= \left[ \frac{n}{n-q+1}\right]
 $$

Then in our notation a consequence of the smoothing theorem is.

\begin{theorem}
Let $U$ be a complex manifold of dimension $n.$ Let $\rho$ be an exhaustion function which is strictly $q$-convex with corners. Then there is a smooth exhaustion function $\tilde{\rho}$ which is $\tilde{q}$-convex,  
$$\tilde{q}= \left[ \frac{n}{n-q+1}    \right].$$
\end{theorem}

Hence in order to prove Theorem 1.8 we just use the smoothing.
We need
 $$ \left[ \frac{n}{n-q+1}\right] \geq 3.$$
 
 This is equivalent to $\frac{n}{n-q+1} \geq 3$ so
 $n\geq 3n-3q+3$ so $2n+3\leq 3q$ so 
 $q\geq \frac{2n+3}{3}.$ Then Theorem 1.8 is a consequence of Theorem 1.7.

Observe that for $n \geq 6$ we get values of $q$ strictly smaller than $n.$

\section{Proof of Theorem 1.6}

\begin{lemma}
Let $L$ be a smooth complex manifold of dimension $q$ in an open set $W\subset \mathbb P^n.$
Suppose that $p\in L.$ Then there is a  neighborhood of $p$, $W' \subset \subset W,$ such that
$-\log \sin^2 d(z,L)$ is  $\mathcal C^\infty$ strictly $q+1-$convex in $W'.$ Moreover if $K$ is a compact subset of $W' \setminus L$, then if $L'$ is close enough to $L$ then the positive eigenvalues of $-\log \sin^2 d (z,L')$ are uniformly bounded below by a positive constant on $K.$
\end{lemma}

\begin{proof} The proof is classical, but we include it for the benefit of the reader and because we need
the estimates to be uniform in continuous variations of $L.$
Recall, Barth  [2], that if $x,y \in \mathbb C^{n+1}\setminus \{0\}$ and the corresponding points in $\mathbb P^n$ are $x^*,y^*$, then the distance $d(x^*,y^*)$ in the Fubini-Study metric is given by the formula:

$$\sin^2 (d(x^*,y^*))= 1-\frac{|<x,\overline{y}>|^2}{|x|^2|y|^2}
$$
\noindent where $<x,\overline{y}>=\sum x_i \overline{y}_i.$
So we are going to estimate the distance to
 a local piece of a smooth $q-$dimensional complex manifold $L.$ We choose a smooth collection of rotations $A_p$ of $\mathbb C^{n+1}$ for $p^*\in L$
so that $p^*=[0: \cdots : 0 : 1]$ and $L$ is given by $[z_1:\cdots :z_q: f^{q+1}_p(z_1, \dots,z_q): \cdots : f^n_p(z_1, \dots,z_q): 1]$
and the functions $f^j_p$ vanish to second order at the origin.
Let $N_L^p$ denote the complex normal plane $\{[0:\cdots :0: z_{q+1}: \cdots : z_n: 1] \}.$ We write $z'=(z_1,\dots,z_q),
z''=(z_{q+1},\dots, z_n)$. Also write $f_p(z')= (f^{q+1}_p(z_1, \dots,z_q), \dots , f^n_p(z_1, \dots,z_q)).$
Set $x=(z',f_p(z'),1)$ and $y=(0,z'',1).$
According to the above formula,
\bea
\sin^2 d(x^*,y^*) & = & 1- \frac{|1+ \sum_{j=q+1}^n \overline{z}_j f^j_p(z') |^2}{(1+|z'|^2+|f_p(z')|^2)(1+|z''|^2)}\\
& \geq & 1- \frac{1+C|z''||z'|^2}{(1+|z'|^2+|f_p(z')|^2)(1+|z''|^2)}\\
& \geq & 1-\frac{1}{1+|z''|^2}\\
& = & \sin^2 d(0, y^*).\\
\eea
To estimate the Levi form of a function $\rho$ we are going to use  that if $\rho\geq \psi$ and $\rho(p)=\psi(p)$ then the Levi form of $\rho$ at $p$ is larger than the Levi form of $\psi$ at $p.$ So we have to find appropriate functions $\psi.$
The above inequality shows that the points on $N_L^p$ are closest points to $p\in L$ and hence that the distance function to $L$ is smooth.
Next we estimate the Levi form of $\sin^2 d(L,y^*).$ Let $y= (z',z'',1)$, $y_0=(0,z_0'',1)$ and pick the point $x=(z',f_p(z'),1)$ on $L.$
We want to estimate the behaviour of $\sin^2 d(L,y^*)$ on the plane given by $(z',tz_0'',1), z'\in \mathbb C^q, t\in \mathbb C$
with $z'$ close to $0$ and $t $ close to $1$. So this is a $q+1-$dimensional plane through $y_0^*.$
We then have
\bea
A&:=&\sin^2 d(L, (z',tz_0'',1))- \sin^2 d((0,0,1),(0,z''_0,1)) \\
& \leq & \sin^2 d((z',f_p(z'),1), (z',tz_0'',1))- \sin^2 d((0,0,1),(0,z''_0,1)) \\
& = & [1-\frac{||z'|^2+ t\sum_{j={q+1}}^n\overline{f}^j_p(z') (z_0'')_j+1  |^2}{(|z'|^2+|f_p|^2+1)(|z'|^2+|t|^2 |z_0''|^2+1)}]\\
& - & [1-\frac{1}{(1+|z''_0|^2)}]\\
& = & \frac{(|z'|^2+|f_p|^2+1)(|z'|^2+|t|^2 |z_0''|^2+1)}
{(|z'|^2+|f_p|^2+1)(|z'|^2+|t|^2 |z_0''|^2+1)(1+|z''_0|^2)}\\
& -& \frac{||z'|^2+ t\sum_{j={q+1}}^n\overline{f}^j_p(z') (z_0'')_j+1  |^2(1+|z''_0|^2)}
{(|z'|^2+|f_p|^2+1)(|z'|^2+|t|^2 |z_0''|^2+1)(1+|z''_0|^2)}\\
& = & \mathcal O((z')^4,(z')^2(t-1))+\frac{|z_0''|^2(|t|^2-1)}
{(|t|^2 |z_0''|^2+1)(1+|z''_0|^2)}\\
& - & \frac{2\Re \sum_{j={q+1}}^n\overline{f}^j_p(z') (z_0'')_j}
{(1+|z''_0|^2)}\\
& - & \frac{|z_0''|^2|z'|^2}
{(1+|z''_0|^2)^2}\\
\eea
Write $\delta= \frac{(|t|^2-1)|z''_0|^2}{1+|z''_0|^2},$ then
\bea
\frac{1}{(|t|^2 |z_0''|^2+1)} & = &  \frac{1}{(1+|z''_0|^2)+(|t|^2-1)|z''_0|^2}\\
& = & \frac{1}{(1+|z''_0|^2)(1+\delta)} \\
& = & \frac{1}{1+|z_0''|^2} (1-\delta+\delta^2 +\cdots)\\.
\eea
We get then, if $\mathcal PO$ refers to pluriharmonic error terms.
\bea
A & \leq & \mathcal O((z')^4,(z')^2(t-1),|t-1|^3) + \frac{|z_0''|^2(|t|^2-1)}{(1+|z_0''|^2)^2} [1-\frac{(|t|^2-1)|z''_0|^2}{1+|z''_0|^2}]\\
& + & \mathcal PO((z')^2) -\frac{|z_0''|^2|z'|^2}
{(1+|z''_0|^2)^2}\\
\eea
So, 
\bea 
\sin^2 d(L, (z',tz_0'',1)) &  \leq & [1-\frac{1}{(1+|z''_0|^2)}]+\mathcal O((z')^4,(z')^2(t-1),|t-1|^3) \\
& + &  \frac{|z_0''|^2(|t|^2-1)}{(1+|z_0''|^2)^2} [1-\frac{(|t|^2-1)|z''_0|^2}{1+|z''_0|^2}]\\
& + &   {\mathcal {P}}{\mathcal{O}} ((z')^2) -\frac{|z_0''|^2|z'|^2}
{(1+|z''_0|^2)^2}.\\
\eea
Next we need to compose with $-\log (x)$. Using the comparison trick, we first consider the $z'$ direction: Let $A-B|z'|^2$
be the right hand side of the previous inequality. Define
\bea 
r & := & -\log [A-B|z'|^2].\; {\mbox{Then}}\\
r_{z} &= & \frac{B\overline{z}}{1-B|z'|^2}\; {\mbox{and}}\\
r_{z\overline{z}}(0) & = & B >0\\
\eea
For the $t$ direction at $t=1,$ we decompose $r$ differently,
\bea
r & := & -\log (A+B(|t|^2-1)-C (|t|^2-1)^2)\\
r_t & = & -\frac{B \overline{t}-2C(|t|^2-1)\overline{t}}{ (A+B(|t|^2-1)-C (|t|^2-1)^2)}\\
r_{t\overline{t}}(1) & = &  \frac{[B][B]-[B-2C][A]}{ A^2}\\
& = & \frac{B^2-AB+2AC}{A^2}\\
\eea
\bea
B^2-AB+2AC & = &[ \frac{|z_0''|^2}{(1+|z_0''|^2)^2} ]^2-  [1-\frac{1}{(1+|z''_0|^2)}] \frac{|z_0''|^2}{(1+|z_0''|^2)^2}\\
&  + & 2[1-\frac{1}{(1+|z''_0|^2)}] \frac{|z_0''|^4}{(1+|z_0''|^2)^3} \\
\eea
\bea
& = & \frac{|z_0''|^4}{(1+|z_0''|^2)^4} -  \frac{|z_0''|^4}{(1+|z_0''|^2)^3}  + 2 \frac{|z_0''|^6}{(1+|z_0''|^2)^4} \\
& = &  \frac{|z_0''|^4}{(1+|z_0''|^2)^4}[1-(1+|z''|^2)+2|z_0''|^2]\\
& = &  \frac{|z_0''|^6}{(1+|z_0''|^2)^4}>0\\
\eea

These calculations show that $-\log \sin^2 d(L,y)$ is strictly plurisubhamonic on a $q+1-$dimensional plane in a neighborhood of $L.$
\end{proof}

Observe that on compact subsets of $W'\setminus L$ any small $\mathcal C^2$ perturbation of $-\log \sin^2 d(z,L):= \rho(z,L)$ is still $q+1$-convex. 
We have a lower bound on the eigenvalues of the Levi form if $L'$ is close to $L.$
In a tubular neighborhood $T$ of $X$ we have 
$$-\log \sin^2 d(z,X)= \sup_{i \in I} -\log \sin^2 d(z,L_i)$$ with $L_i$ a countable family of plaques. 

\medskip

We are ready to finish the proof of Theorem 1.6.
\begin{proof}
Fix an affine chart $U$ of $\mathbb P^n$. 
Let $T$ be a tubular neighborhood of $X$ in $\mathbb P^n.$ Fix $p\in X.$ It is possible to find real hyperplanes $P_1,P_2$ in the Euclidean sense and distant of $R$ large enough, so that in a neighborhood of $p$ in $T \cap U$, we have $d(z,X)=\min_i d(z,L_i)$ with $L_i$ a countable family of the intersection of leaves with the region between $P_1$ and $P_2.$ Here $d$ denotes the distance in the Fubini Study metric. Let $\rho(z,L_i)=
-\log \sin^2 d(z,L_i),$ the functions constructed in Lemma 3.1, associated to the complex manifold $L_i$. We can write $T=\cup T_j$, $T_j= \{z;
\frac{1}{2^{j+1}} \leq d(z,X)< \frac{1}{2^j}\}.$ 

Let $\chi_j$ be a smooth function supported on $T_{j+1} \cup T_j \cup T_{j-1}$ with value $1$ on $T_j.$ Let $\alpha(t):= -\log \sin^2 t$ for $t>0.$ 

Using the uniformity of Lemma 4.1 it is possible to find finitely many $L_i's, i\in I_j$ so that on a neighborhood of $\overline{T}_j$, the function $\max_{i\in I_j} (\rho(z,L_i)+c_j \chi_j)$ is larger than $\alpha(d(z,X))$ and is strictly $q+1-$ convex with corners. This permits to construct a function $\rho$ which is strictly $q+1-$ convex with corners and which is a small perturbation of the function $\alpha(d(z,X)).$ Observe that on $T_{j+1} \cup T_j \cup T_{j-1},$ $\rho$ is obtained as a perturbation of the above function and we do the construction inductively on $j.$ Once we have obtained this strictly $q+1-$convex function with corners which goes to infinity near $X,$ we apply Peternell's Theorem 1.6 [12].

\end{proof}


\section{Examples in dimension 2}

In dimension 2, it is well known that  we do not have Hartogs' extendability for line bundles. Consider the line bundle in $\mathbb C^2 \setminus \{(0,0)\}$ given in $z \neq 0, w \neq 0$ by the transition function $f(z,w)=\exp (\frac{1}{zw}).$ It does not extend through $(0,0).$ If it did, we would have $f=\frac{f_1}{f_2}$ in a neighborhood of $(0,0)$ with $f_i$ holomorphic in $z \neq 0$ and $f_2$ holomorphic in $w\neq 0.$ Using Laurent series expansion, one checks easily that this is impossible.

We can construct also an example of a line bundle in the complement of the unit ball, with the unit sphere as natural boundary.
Let $\{(z_i,w_i)\}$ denote a dense
set of points on the boundary of the unit ball.
Let $f=\sum_i \epsilon_i \frac{1}{(z-z_i)(w-w_i)}.$ If the $\epsilon_i$ go to zero fast enough,
then this is a well defined holomorphic function on the set $\{1<|z|<2,1<|w|<2\}$ Hence we can use $g=e^f$ as transition function for a holomorphic line bundle on $\{|z|,|w|<2\}\setminus \{|z|\leq 1,|w|\leq 1\}.$ We show that this line bundle extends to $\{|z|,|w|<2\}\setminus \{|z|^2+|w|^2 \leq 1\}$, but that it does not have an extension  across any point in the boundary of the unit ball.

We can divide the $(z_i,w_i)$ into finitely many groups of points, lying in the same small polydisc.
Then for each such small polydisc the corresponding line bundle is holomorphic in the complement
of this polydisc. The line bundle is the product of these. Hence the line bundle extends to the complement of any $B(0,1+\delta).$ We need to show it does not extend through any point $(z_i,w_i)$.
We choose any small polydisc $D$ around $(z_i,w_i)$. containing only points $(z_j,w_j)$ with
$\epsilon_j$ very small compared to $\epsilon_i.$ Cover the rest with polydiscs also. Then the line bundle for each of the other polydiscs extends across $(z_i,w_i)$ so the obstruction to extension comes only from $D.$ Assume for simplicity that $z_i=w_i=0.$ We note that the corresponding additive Cousin problem is a holomorphic function which can be Laurent series expanded and has terms with negative powers in both $z$ and $w$ at the same time. These can not be written as a difference of two holomorphic functions as needed because these can have only poles in one variable at a time. So this shows non-extension.

\section{Levi flat manifolds in open sets in $\mathbb P^n.$}

 Lins Neto [10] has shown the nonexistence of real analytic Levi flat hypersurfaces $\Sigma$ in $\mathbb P^n$ for $n \geq 3$. He uses the extension of the foliation of $\Sigma$ as a holomorphic foliation of codimension $1$ in $\mathbb P^n.$ Siu [13] has shown the non existence of a smooth Levi flat hypersuface in $\mathbb P^n, n \geq 3,$ foliated by leaves of codimension $1.$ M. Brunella (personal communication) has observed that there are no exceptional minimal sets for foliations of dimension $q \geq \frac{n+1}{2}$ in $\mathbb P^n$ as a consequence of the Baum-Bott formula.

We have the following result regarding Levi-flat manifolds in open subsets of $\mathbb P^n$:

\begin{theorem}
Let $\Omega\subset\mathbb P^n$ be an open set, and 
assume that $\mathcal F$ is a $q$-dimensional holomorphic
foliation on $\Omega$ with $q \geq  \frac{2n}{3}.$
Then there 
does not exist a $\mathcal C^1$-smooth $\mathcal F$-invariant
submanifold $\Sigma\subset\mathbb P^n$, $\Sigma\subset\Omega$,
such that $\Sigma=\cap_{j=1}^{n-q}\Sigma_{j}$, where 
$\Sigma_j$ is a hypersurface in $\mathbb P^n$, and the 
normals to the $\Sigma_j$'s are everywhere complex 
linearly independent on $\Sigma$.
\end{theorem}

\begin{proof}
 We assume to get a contradiction that there exists
such a manifold $\Sigma$.  Note that by Theorem 1.6 and  Theorem 3.7 we have that 
$\mathbb P^n\setminus\Sigma$ admits a smooth 3-convex exhaustion 
function $\rho$.

We cover $\Sigma$ by a finite number of balls $B_j$
such that $\mathcal F$ is defined by 
holomorphic $(n-q)$-forms $\omega_j=df_j^1\wedge\cdots\wedge 
df_j^{n-q}$ on $B_j$, and these give rise to a line
bundle $\mathcal N:=(f_{ij})$ with $f_{ij}=\omega_i/\omega_j$.
Note that $\mathcal N$ is holomorphically nontrivial:
a splitting $f_{ij}=f_j/f_i$ would give 
rise to a holomorphic form $\omega:=f_j\cdot\omega_j$
on $\Omega$.  By Hartogs extension, we have that $\omega$
would extend to all of $\mathbb P^n$, but it is 
known that $\mathbb P^n$ does not support a nontrivial 
holomorphic $m$-form
for any $m\geq 1$. 

On the other hand there exist $n-q$ non-vanishing continuous 
vector fields $\xi_j$ on $\Sigma$, each pointing in 
the normal direction to $\Sigma_j$, and such that 
they span an $(n-q)$-dimensional complex space over $\mathbb C$. 
This means that $\psi_j:=\omega_j(\xi_1,...,\xi_{n-q})$
defines a continuous splitting of $(f_{ij})$ over $\Sigma$, 
hence $\mathcal N\rightarrow\Sigma$ is topologically trivial, 
and consequently has roots of all orders.  This holds
also in a neighborhood retract to $\Sigma$. 

By Theorem 3.6 we have that $\mathcal N$ along 
with all its roots $\mathcal N_k$, \emph{i.e.}, $\mathcal N_k^{\otimes k}=\mathcal N$,
extends uniquely to $\mathbb P^n$.  It follows that the 
Chern class of (the extended) $\mathcal N$ is zero, and 
this contradicts the non-triviality of $\mathcal N$ on $\Omega$.
\end{proof}

\begin{remark}
An argument as above also gives the non-existence of real analytic 
Levi-flat hypersurfaces in $\mathbb P^n$ for $n\geq 3$: the Levi 
foliation of such a hypersurface $\Sigma$ extends to a neighborhood 
of $\Sigma$, and by the solution to the Levi-problem we have 
that $\mathbb P^n\setminus\Sigma$ is Stein.  
\end{remark}

\end{document}